\documentclass{amsart}

\usepackage{amsmath}
\usepackage{amssymb}
\usepackage{amsthm}
\usepackage{graphicx}

\newtheorem{theorem}{Theorem}[section]
\newtheorem{lemma}[theorem]{Lemma}
\newtheorem{proposition}[theorem]{Proposition}

\theoremstyle{definition}
\newtheorem{definition}[theorem]{Definition}
\newtheorem{example}[theorem]{Example}

\begin{document}

\title{Knotted handle decomposing spheres for handlebody-knots}

\author{Atsushi Ishii}
\address{Institute of Mathematics, University of Tsukuba, 1-1-1
Tennodai, Tsukuba, Ibaraki 305-8571, Japan}
\curraddr{}
\email{aishii@math.tsukuba.ac.jp}
\thanks{The first author was supported by JSPS KAKENHI Grant Number 21740035.}

\author{Kengo Kishimoto}
\address{Osaka Institute of Technology, 5-16-1 Omiya, Asahi-ku Osaka
535-8585 Japan}
\email{kishimoto@ge.oit.ac.jp}
\thanks{}

\author{Makoto Ozawa}
\address{Department of Natural Sciences, Faculty of Arts and Sciences,
Komazawa University, 1-23-1 Komazawa, Setagaya-ku, Tokyo, 154-8525, Japan}
\email{w3c@komazawa-u.ac.jp}
\thanks{The third author was supported by JSPS KAKENHI Grant Number 23540105.}


\date{}



\begin{abstract}
We show that a handlebody-knot whose exterior is boundary-irreducible
has a unique maximal unnested set of knotted handle decomposing spheres
up to isotopies and annulus-moves.
As an application, we show that the handlebody-knots $6_{14}$ and
$6_{15}$ are not equivalent.
We also show that some genus two handlebody-knots with a knotted handle
decomposing sphere can be determined by their exteriors.
As an application, we show that the exteriors of $6_{14}$ and $6_{15}$
are not homeomorphic.
\end{abstract}

\maketitle

\section{Introduction}

A \textit{genus $g$ handlebody-knot} is a genus $g$ handlebody embedded
in the $3$-sphere $S^3$.
Two handlebody-knots are \textit{equivalent} if one can be transformed
into the other by an isotopy of $S^3$.
A handlebody-knot is \textit{trivial} if it is equivalent to a
handlebody standardly embedded in $S^3$, whose exterior is a
handlebody.
We denote by $E(H)=S^3-\operatorname{int}H$ the exterior of a
handlebody-knot $H$.

\begin{definition} \label{def:n-decomposing sphere}
A $2$-sphere $S$ in $S^3$ is an \textit{$n$-decomposing sphere} for a
handlebody-knot $H$ if
\begin{itemize}
\item[(1)] $S\cap H$ consists of $n$ essential disks in $H$, and
\item[(2)] $S\cap E(H)$ is an incompressible and not boundary-parallel
surface in $E(H)$.
\end{itemize}
\end{definition}

In some cases it might be suitable to replace the condition (2) in
Definition~\ref{def:n-decomposing sphere} with the condition
\begin{itemize}
\item[(2)'] $S\cap E(H)$ is an incompressible, boundary-incompressible,
and not boundary-parallel surface in $E(H)$,
\end{itemize}
although we adopt the condition (2) in this paper.
The two definitions are equivalent if $n=1$, or $n=2$ and $E(H)$ is
boundary-irreducible.

A handlebody-knot $H$ is \textit{reducible} if there exists a
$1$-decomposing sphere for $H$, where we remark that (2) follows from
(1) when $n=1$.
A handlebody-knot is \textit{irreducible} if it is not reducible.
A handlebody-knot $H$ is irreducible if $E(H)$ is boundary-irreducible.
The converse is true for a genus two handlebody-knot
$H$.
In particular, for a genus two handlebody-knot $H$, the following are
equivalent:
\begin{itemize}
\item[(1)] $H$ is irreducible.
\item[(2)] $\pi_1(E(H))$ is indecomposable with respect to free products.
\item[(3)] $E(H)$ is boundary-prime.
\item[(4)] $E(H)$ is boundary-irreducible.
\end{itemize}
By \cite{Tsukui75}, we have the equivalence between (1) and (2).
By \cite{Jaco69}, we have the equivalence between (2) and (3) for a
handlebody-knot $H$ of arbitrary genus.
The conditions (3) and (4) are equivalent if $E(H)$ is not a solid torus
(cf.~\cite[Proposition~2.15]{Suzuki75}).
We remark that there is an irreducible genus $g\neq2$ handlebody-knot
whose exterior is not boundary-irreducible
(cf.~\cite[Theorem~5.4]{Suzuki75}).


The decomposition by $1$-decomposing spheres is unique for a trivial
handlebody-knot
and a genus two handlebody-knot~\cite{Tsukui70}.
The uniqueness is not known for a genus $g\geq3$ handlebody-knot.

\begin{definition}
A $2$-sphere $S$ in $S^3$ is a
\textit{knotted handle decomposing sphere} for a handlebody-knot $H$ if
\begin{itemize}
\item[(1)]
$S\cap H$ consists of two parallel essential disks in $H$, and
\item[(2)]
$S\cap E(H)$ is an incompressible and not boundary-parallel surface in
$E(H)$.
\end{itemize}
\end{definition}

We say that a $2$-sphere $S$ bounds $(B,K;H)$ if $S$ bounds a $3$-ball
$B$ so that $S\cap H$ consists of two parallel essential disks in $H$,
and that $H\cup E(B)$ is equivalent to a regular neighborhood of a
nontrivial knot $K$.
A knotted handle decomposing sphere for $H$ bounds $(B,K;H)$.
A $2$-sphere $S$ which bounds $(B,K;H)$ is not always a knotted handle
decomposing sphere for $H$
(see the left picture of Figure~\ref{fig:not knotted}).
In this paper, we represent a handlebody-knot by a spatial trivalent
graph whose regular neighborhood is the handlebody-knot as shown in
Figure~\ref{fig:not knotted}.
Then the intersection of the spatial trivalent graph and the $2$-sphere
indicates two disks.

If $H$ is a genus $g\geq2$ handlebody-knot whose exterior is
boundary-irreducible, then a $2$-sphere $S$ which bounds $(B,K;H)$ is a
knotted handle decomposing sphere for $H$, where we note that $g\geq2$
implies that $S\cap E(H)$ is not boundary-parallel in $E(H)$, and that
the boundary-irreducibility implies the incompressibility of
$S\cap E(H)$.
A trivial handlebody-knot has no knotted handle decomposing sphere,
since the decomposition by $1$-decomposing spheres is unique for a
trivial handlebody-knot.

\begin{figure}
\begin{center}
\includegraphics[scale=0.6]{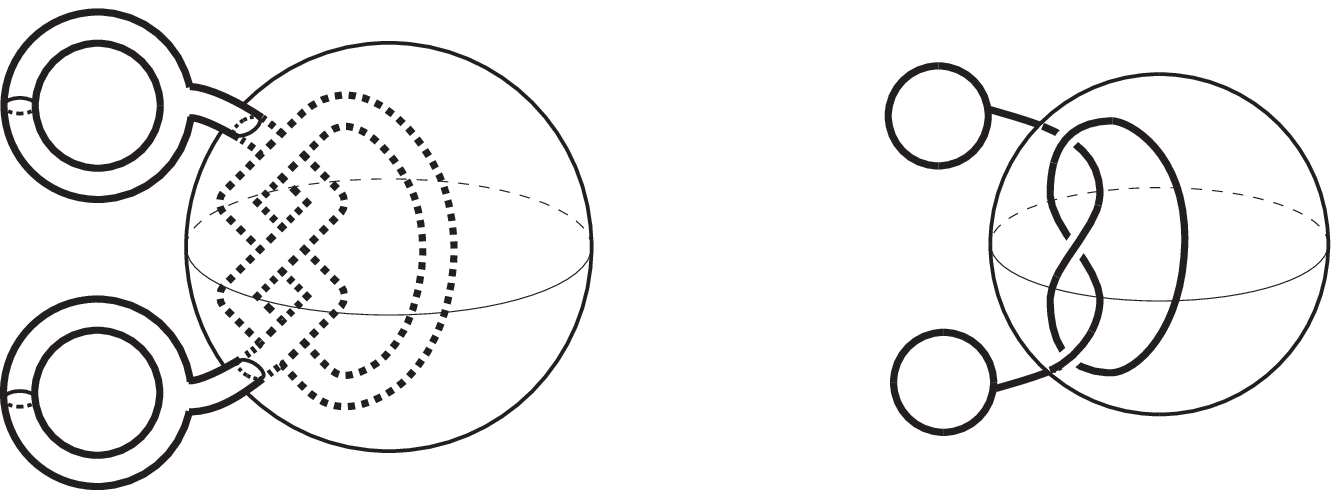}
\end{center}
\caption{}
\label{fig:not knotted}
\end{figure}

In~\cite{IshiiKishimotoMoriuchiSuzuki12}, Moriuchi, Suzuki and the first
and second authors gave a table of genus two handlebody-knots up to six
crossings, and classified them according to the crossing number and the
irreducibility.
There are three pairs of handlebody-knots whose fundamental groups are
isomorphic in the table.
S.~Lee and J.~H.~Lee~\cite{LeeLee12} gave inequivalent genus two
handlebody-knots with homeomorphic exteriors including the two pairs
$5_1$, $6_4$ and $5_2$, $6_{13}$ in the table, and distinguish them by
classifying essential surfaces in the exteriors.
We note that Motto~\cite{Motto90} gave different examples with
homeomorphic exteriors which do not appear in the above table.

The pair $6_{14}$, $6_{15}$ is the remaining pair of handlebody-knots
whose fundamental groups are isomorphic.
In Section~\ref{sec:decomposition}, we show that a handlebody-knot whose
exterior is boundary-irreducible has a unique maximal unnested set of
knotted handle decomposing spheres up to isotopies and annulus-moves.
As an application, we show that the handlebody-knots $6_{14}$ and
$6_{15}$ are not equivalent.
In Section~\ref{sec:hdb-knots and their exteriors}, we show that some
genus two handlebody-knots with a knotted handle decomposing sphere can
be determined by their exteriors.
As an application, we show that the exteriors of the handlebody-knots
$6_{14}$ and $6_{15}$ are not homeomorphic.

\section{A unique decomposition for a handlebody-knot}
\label{sec:decomposition}

Let $H$ be a handlebody-knot in $S^3$, and $S$ a knotted handle
decomposing sphere for $H$ which bounds $(B,K;H)$.
Let $A$ be an annulus properly embedded in $E(H)-\operatorname{int}B$ so
that $A\cap S=l$ is an essential loop in the annulus $S\cap E(H)$, and
that $A\cap\partial H=l'$ bounds an essential disk $D$ in $H$, where
$\partial A=l\cup l'$ (see Figure~\ref{fig:annulus-move}).
Put $T=(S\cap E(H))\cup(B\cap\partial H)$.
Let $A'$ be an annulus obtained from $T$ by cutting along $l$ and
pasting two copies of $A$, where $T$ is slightly isotoped so that
$T\cap H=\emptyset$.
Then we have a new knotted handle decomposing sphere $S'$ obtained from
$A'$ by attaching two parallel copies of $D$ to $\partial A'$.
We say that $S'$ is obtained from $S$ by an \textit{annulus-move} along
$A$.
For example, in Figure~\ref{fig:annulus-move (example)}, $S'$ is
obtained from $S$ by an annulus-move along $A$, where we note that $S$
and $S'$ are not isotopic in the exterior of the handlebody-knot.

\begin{figure}[htbp]
\includegraphics[trim=0mm 0mm 0mm 0mm, width=.6\linewidth]{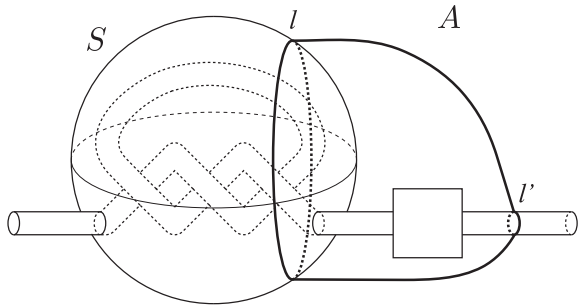}
\caption{An annulus-move along $A$}
\label{fig:annulus-move}
\end{figure}

\begin{figure}
\includegraphics[scale=0.6]{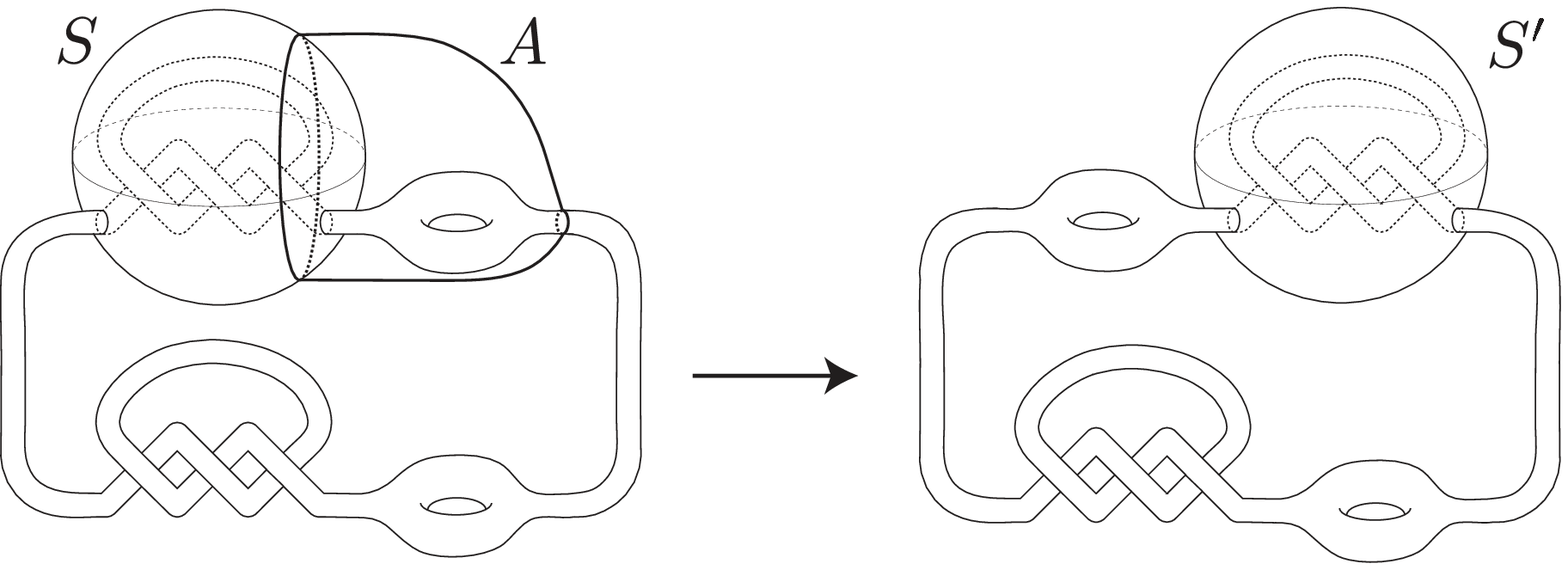}
\caption{}
\label{fig:annulus-move (example)}
\end{figure}

A set $\mathcal{S}=\{S_1, \ldots,S_n\}$ of knotted handle decomposing
spheres for a handlebody-knot $H$ is \textit{unnested} if each sphere
$S_i$ bounds $(B_i,K_i;H)$ so that $B_i\cap B_j=\emptyset$ for
$i\ne j$.
An unnested set $\mathcal{S}$ is \textit{maximal} if $n\geq m$ for any
unnested set $\{S'_1,\ldots,S'_m\}$ of knotted handle decomposing
spheres for $H$.
By the Haken--Kneser finiteness theorem~\cite{Haken68,Kneser29}, there
exists a maximal unnested set of knotted handle decomposing spheres for
$H$.
By Schubert's theorem \cite{Schubert49}, $K_i$ is prime for any $i$ if
$\mathcal{S}$ is maximal.

\begin{lemma} \label{lem:disjoint khds}
Let $H$ be a handlebody-knot whose exterior is boundary-irreducible.
Let $\mathcal{S}=\{S_1,\ldots,S_n\}$ be an unnested set of knotted
handle decomposing spheres for $H$ such that $S_i$ bounds $(B_i,K_i;H)$
and that $K_i$ is prime for any $i$.
Let $\mathcal{S}'=\{S'_1,\ldots,S'_m\}$ be a set of 2-decomposing
spheres for $H$.
Then $\mathcal{S}$ can be deformed so that $S_i\cap S_j'=\emptyset$ for
any $i,j$ by isotopies and annulus-moves.
\end{lemma}

\begin{proof}
Put $A_i=S_i\cap E(H)$ for $i=1,\ldots,n$ and $A'_j=S'_j\cap E(H)$ for
$j=1,\ldots,m$.
We may assume that $A_i\cap A'_j$ consists of essential arcs or loops in
both $A_i$ and $A'_j$, and that $|A_i\cap A'_j|$ is minimal up to
isotopies and annulus-moves.

In the case that $A_i\cap A'_j$ consists of essential arcs for some $i$
and $j$, let $\Delta$ be a component of $A'_j\cap B_i$ which is
cobounded by two adjacent arcs of $A_i\cap A'_j$ in $A'_j$.
By the minimality of $|A_i\cap A'_j|$, $\partial\Delta$ winds around
$B_i-\operatorname{int}H$ longitudinally twice, where we note that
$B_i-\operatorname{int}H$ is the exterior of $K_i$.
Hence $\Delta$ is an essential disk in $B_i-\operatorname{int}H$, which
implies that $B_i-\operatorname{int}H$ is boundary-reducible and $K_i$
is a trivial knot, a contradiction.

In the case that $A_i\cap A'_j$ consists of essential loops for some $i$
and $j$, let $F$ be an outermost subannulus of $A'_j$ which is cut by
$A_i\cap A'_j$.
If $F$ is contained in $B_i$, then by the primeness of $K_i$, we can
isotope off $F$ from $B_i$.
Hence $F$ is in the outside of $B_i$.
Then by an annulus move for $S_i$ along the annulus $F$, we can reduce
$|A_i\cap A'_j|$.
This contradicts to the minimality of $|A_i\cap A'_j|$.
\end{proof}

\begin{theorem} \label{thm:decomposition}
A handlebody-knot $H$ whose exterior is boundary-irreducible has a
unique maximal unnested set of knotted handle decomposing spheres for
$H$ up to isotopies and annulus-moves.
\end{theorem}

\begin{proof}
Let $\mathcal{S}=\{S_1,\ldots,S_n\}$, $\mathcal{S}'=\{S'_1\ldots,S'_n\}$
be maximal unnested sets of knotted handle decomposing spheres for $H$
such that $S_i$ and $S_j'$ bound $(B_i,K_i;H)$ and $(B'_j,K'_j;H)$,
respectively.
By Lemma~\ref{lem:disjoint khds}, we can deform $\mathcal{S'}$ so that
$S_i\cap S_j'=\emptyset$ for any $i,j$ by isotopies and annulus-moves.
We also deform $\mathcal{S'}$ so that $B_i\cap B'_j=\emptyset$ by
isotopies if $B_i\cap B'_j$ is homeomorphic to $S^2\times I$, where $I$
is an interval.
Then we have $B_i\subset B'_j$, $B'_j\subset B_i$, or
$B_i\cap B_j'=\emptyset$ for any $i,j$.
Since $\mathcal{S}'$ is maximal, for any $B_i$, there exists a $3$-ball
$B_j'$ such that $B_i\subset B'_j$ or $B'_j\subset B_i$.
Since $K_i$ and $K_j'$ are prime, $S_i$ is parallel to $S'_j$.
This gives a one-to-one correspondence between $\mathcal{S}$ and
$\mathcal{S}'$.
Hence a maximal unnested set of knotted handle decomposing spheres for
$H$ is unique up to isotopies and annulus-moves.
\end{proof}

\begin{proposition} \label{prop:khds with ball}
Let $H$ be a genus $g$ handlebody-knot whose exterior is
boundary-irreducible.
Let $\{S_1,\ldots,S_n\}$ be an unnested set of knotted handle
decomposing spheres for $H$ such that $S_i$ bounds $(B_i,K_i;H)$ for any
$i$.
Put $H':=H\cup B_{m+1}\cup\cdots\cup B_n$.
Then $\{S_1,\ldots,S_m\}$ is an unnested set of knotted handle
decomposing spheres for $H'$, or $g=1$ and $m=1$.
\end{proposition}

\begin{proof}
Suppose that $S_i\in\{S_1,\ldots,S_m\}$ is not a knotted handle
decomposing sphere for $H'$.
If $S_i\cap E(H')$ is compressible in $E(H')$, then $S_i\cap E(H)$ is
also compressible in $E(H)$, a contradiction.
If $S_i\cap E(H')$ is parallel to an annulus $A\subset \partial E(H')$
in $E(H')$, then $A$ contains some annuli of
$(B_{m+1}\cup\cdots\cup B_n)\cap\partial H'$.
This shows that $g=1$ and $m=1$.
\end{proof}

\begin{proposition} \label{prop:b-irreducible with ball}
Let $H$ be a genus $g\geq2$ handlebody-knot, $S$ a $2$-sphere which
bounds $(B,K;H)$.
If $E(H\cup B)$ is boundary-irreducible, then so is $E(H)$.
\end{proposition}

\begin{proof}
Suppose that $E(H)$ is boundary-reducible and let $D$ be a compressing
disk in $E(H)$.
Since $E(H\cup B)$ is boundary-irreducible, $D$ intersects with the
annulus $A=S\cap E(H)$.
Since $E(H)$ is irreducible, we may assume that $D\cap A$ consists of
essential arcs in $A$.
Since the knot $K$ is nontrivial, an outermost disk of $D$ gives a
compressing disk in $E(H\cup B)$.
This is a contradiction.
\end{proof}

A \textit{($n$-component) handlebody-link} is a disjoint union of $n$
handlebodies embedded in the $3$-sphere $S^3$.
A \textit{non-split handlebody-link} is a handlebody-link $H$ which has
no disjoint $3$-balls $B_1,B_2$ such that $H\cap B_1\neq\emptyset$,
$H\cap B_2\neq\emptyset$.

\begin{proposition} \label{prop:b-irreducible without ball}
Let $H$ be a handlebody-knot, $S$ a $2$-sphere which bounds $(B,K;H)$.
Suppose that $H-\operatorname{int}B$ is a non-split handlebody-link
whose exterior is boundary-irreducible.
If $H-\operatorname{int}B$ is $2$-component handlebody-link or
$E(H\cup B)$ is a handlebody, then $E(H)$ is boundary-irreducible.
\end{proposition}

\begin{proof}
Suppose that $E(H)$ is boundary-reducible.
Let $D$ be a compressing disk in $E(H)$.
Put $A=S\cap E(H)$.
If $D\cap A\neq\emptyset$, then we may assume that $D\cap A$ consists of
essential arcs in $A$, since $E(H)$ is irreducible.
Since the knot $K$ is nontrivial, an outermost disk $\delta$ of $D$ is
contained in $E(H\cup B)$.
If $H-\operatorname{int}B$ is not a handlebody-knot, then the arc
$\delta\cap(H-\operatorname{int}B)$ connects the different components of
$H-\operatorname{int}B$ on $\partial(H-\operatorname{int}B)$, a
contradiction.
If $E(H\cup B)$ is a handlebody, then $\delta$ cuts $E(H\cup B)$ into
$E(H-\operatorname{int}B)$, which is a handlebody.
This implies that $H-\operatorname{int}B$ is trivial, which contradicts
that $E(H-\operatorname{int}B)$ is boundary-irreducible.
Then $D\cap A=\emptyset$, and so $D$ is in $E(H-\operatorname{int}B)$.
Since $E(H-\operatorname{int}B$) is boundary-irreducible, $D$ is
inessential in $E(H-\operatorname{int}B)$.
Let $D'$ be a disk in $\partial E(H-\operatorname{int}B)$ such that
$\partial D'=\partial D$.

Let $D_1,D_2$ be the disks such that $S\cap H=D_1\cup D_2$.
If $D'\cap(D_1\cup D_2)=\emptyset$, then $\partial D'$ is inessential in
$\partial E(H)$, which contradicts that $D$ is essential in $E(H)$.
If $D'\cap(D_1\cup D_2)=D_1$ or $D'\cap(D_1\cup D_2)=D_2$, then the
$2$-sphere $S'=D'\cup D$ can be slightly isotoped so that
$S'\cap(H-\operatorname{int}B)=\emptyset$, which contradicts that
$H-\operatorname{int}B$ is non-split, since $S'$ separates $D_1$ and
$D_2$.
Thus $D_1,D_2\subset D'$.
If $H-\operatorname{int}B$ is not a handlebody-knot, then $D'$ connects
the different components of $H-\operatorname{int}B$ on
$\partial(H-\operatorname{int}B)$, a contradiction.
If $E(H\cup B)$ is a handlebody, then the $2$-sphere $S'=D'\cup D$ can
be slightly isotoped so that $D'$ is properly embedded in
$H-\operatorname{int}B$.
Then $S'$ separates a handlebody $E(H\cup B)$ into a solid torus and a
handlebody which is the exterior of $H-\operatorname{int}B$.
This contradicts that $H-\operatorname{int}B$ is nontrivial.
\end{proof}

\begin{example} \label{ex:5_4,6_14,6_15}
We show that any two of the handlebody-knots $5_4$, $5_4^*$, $6_{14}$,
$6_{14}^*$, $6_{15}$ and $6_{15}^*$ are not equivalent, where $5_4$,
$6_{14}$ and $6_{15}$ are the handlebody-knots depicted in
Figure~\ref{fig:5_4,6_14,6_15withS}, and $5_4^*$, $6_{14}^*$ and
$6_{15}^*$ are their mirror images, respectively.

Let $H$ be one of the handlebody-knots $5_4$, $5_4^*$, $6_{14}$,
$6_{14}^*$, $6_{15}$ and $6_{15}^*$.
Let $S$ be the knotted handle decomposing sphere for $H$ depicted in
Figure~\ref{fig:5_4,6_14,6_15withS}, where $S$ bounds $(B,K;H)$ and $K$
is a trefoil knot.
By Proposition~\ref{prop:b-irreducible without ball}, $E(H)$ is
boundary-irreducible.
By Proposition~\ref{prop:khds with ball}, $\{S\}$ is a maximal unnested
set of knotted handle decomposing spheres for $H$, since the trivial
handlebody-knot $H\cup B$ has no knotted handle decomposing sphere.
Then $S$ is unique by Theorem~\ref{thm:decomposition}, which implies
that the pair $(K,H-\operatorname{int}B)$ is an invariant of $H$.
Hence any two of the handlebody-knots $5_4$, $5_4^*$, $6_{14}$,
$6_{14}^*$, $6_{15}$ and $6_{15}^*$ are not equivalent.
\end{example}

\begin{figure}
\begin{center}
\includegraphics{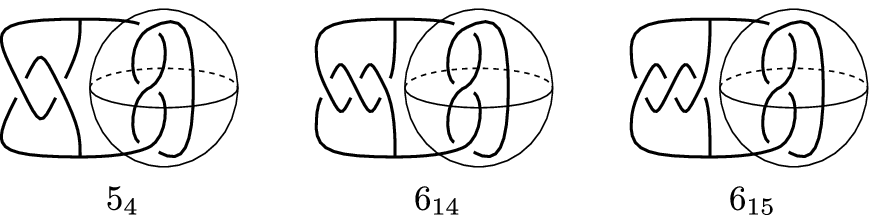}
\end{center}
\caption{}
\label{fig:5_4,6_14,6_15withS}
\end{figure}

\begin{proposition}
There exists a sequence of handlebody-knots $H_n$
$(n\in\mathbb{N}\cup\{0\})$ satisfying the following conditions.
\begin{itemize}
\item
$H_0$ is the trivial genus two handlebody-knot,
which has no knotted handle decomposing sphere.
\item
$H_n$ has a unique knotted handle decomposing sphere
$S_n$ which bounds $(B_n,K_n;H_n)$ for $n\geq1$.
\item
$H_n\cup B_n$ is equivalent to $H_{n-1}$ as a handlebody-knot.
\end{itemize}
\end{proposition}

\begin{figure}[htbp]
\begin{center}
\includegraphics[scale=0.4]{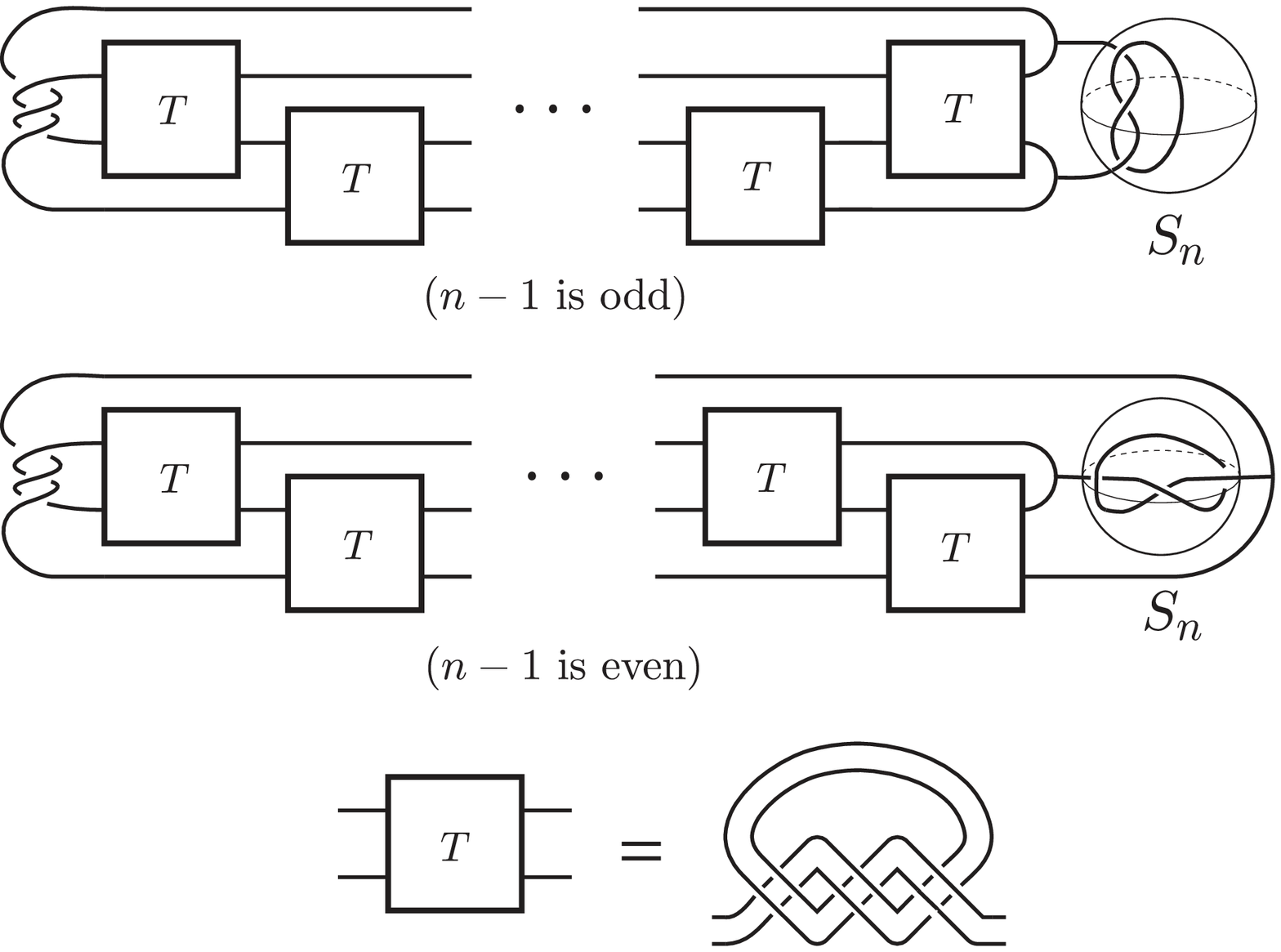}
\end{center}
\caption{}
\label{fig:H_n}
\end{figure}

\begin{proof}
Let $H_0$ be the trivial genus two handlebody-knot.
For $n\geq1$, let $H_n$ be the genus two handlebody-knot with $n-1$
tangles $T$ and a $2$-sphere $S_n$ bounding $(B_n,K_n;H_n)$ as depicted
in Figure~\ref{fig:H_n}.
Then $H_n\cup B_n$ is equivalent to $H_{n-1}$.
We remark that $H_1$ is the irreducible handlebody-knot $6_{14}$.
It follows by Proposition~\ref{prop:b-irreducible with ball} that $H_n$
is irreducible for $n\geq1$.
Then $S_n$ is a knotted handle decomposing sphere for $H_n$.

We prove by induction on $n$ that $S_n$ is a unique knotted handle
decomposing sphere for $H_n$.
We already showed that $S_1$ is a unique knotted handle decomposing
sphere for $H_1$ in Example~\ref{ex:5_4,6_14,6_15}.
Assume that $S_{n-1}$ is a unique knotted handle decomposing sphere for
$H_{n-1}$.
Suppose that $S_n$ is not a unique knotted handle decomposing sphere for
$H_n$.
Then, by Lemma~\ref{lem:disjoint khds} and
Theorem~\ref{thm:decomposition}, there is a knotted handle decomposing
sphere $S_n'$ for $H_n$ which bounds $(B_n',K_n';H_n)$ such that the set
$\{S_n,S_n'\}$ is a maximal unnested set of knotted handle decomposing
spheres for $H_n$.

Let $K_n^-$ be the core of $H_n-\operatorname{int}B_n$, which is a
satellite knot.
Let $T'$ be the tangle obtained from $T$ and $3$ half twists as the
leftmost tangle of $K_n^-$ in Figure~\ref{fig:H_n}.
Then $T$ and $T'$ are prime tangles (cf.~\cite{HayashiMatsudaOzawa99}).
Since $K_n^-$ is obtained from $T'$ and $n-2$ copies of $T$ by tangle
sum, $K_n^-$ is a prime knot~\cite{Lickorish81}.
It follows by Proposition~\ref{prop:khds with ball} that $S'_n$
corresponds to $S_{n-1}$.
Hence $K_n'$ is the positive trefoil knot, and
$(H_n\cup B_n)-\operatorname{int}B'_n$ is a regular neighborhood of
$K_{n-1}^-$.
A loop $l$ of $S'_n\cap\partial H_n$ is in
$\partial(H_n-\operatorname{int}B_n)$, since the set $\{S_n,S_n'\}$ is
unnested.

If $l$ is essential in $\partial(H_n-\operatorname{int}B_n)$, then $l$
is a meridian loop of a solid torus $H_n-\operatorname{int}B_n$.
By the primeness of $K_n^-$, the positive trefoil knot $K_n'$ is
equivalent to the satellite knot $K_n^-$ for $n>1$, a contradiction.

If $l$ is inessential in $\partial(H_n-\operatorname{int}B_n)$, then $l$
bounds a disk $D$ in $\partial(H_n-\operatorname{int}B_n)$.
Let $D_1,D_2$ be the disks such that $S_n\cap H_n=D_1\cup D_2$.
Since $l$ is essential in $\partial H_n$,
$D\cap(D_1\cup D_2)\neq\emptyset$.
If $D$ contains both $D_1$ and $D_2$, then $l$ is a separating loop in
$\partial H_n$ and $\partial H_{n-1}$, which contradicts that
$S_{n-1}\cap\partial H_{n-1}$ consists of non-separating disks.
Thus $D$ contains either $D_1$ or $D_2$, which implies that
$l$ is parallel to the loops of $S_n\cap\partial H_n$.
Then $H_n-\operatorname{int}B_n$ and
$(H_n\cup B_n)-\operatorname{int}B'_n$ are equivalent as
handlebody-knots.
It follows that $K_n^-$ and $K_{n-1}^-$ are equivalent, which
contradicts that $K_m^-$ has a non-trivial Fox $3$-coloring if and only
if $m$ is odd, since the replacement of the tangle $T$ with the trivial
tangle does not change the number of Fox $3$-colorings.

Therefore $S_n$ is a unique knotted handle decomposing sphere for
$H_n$.
This completes the proof.
\end{proof}

\section{Handlebody-knots and their exteriors}
\label{sec:hdb-knots and their exteriors}

In this section, we show that some genus two handlebody-knots with a
knotted handle decomposing sphere can be determined by their exteriors.
As an application, we show that the exteriors of the handlebody-knots
$6_{14}$ and $6_{15}$ are not homeomorphic.
  
\begin{theorem} \label{thm:hdb-knots and their exteriors}
For $i=1,2$, let $H_i$ be an irreducible genus two handlebody-knot with
a knotted handle decomposing sphere $S_i$ bounding $(B_i,K_i;H_i)$ such
that $B_i$ contains all spheres in a maximal unnested set of knotted
handle decomposing spheres for $H_i$.
Suppose that $E(H_i\cup B_i)$ is a handlebody and that
$H_i-\operatorname{int}B_i$ is a nontrivial handlebody-knot for
$i=1,2$.
Then $H_1$ and $H_2$ are equivalent if and only if there is an
orientation preserving homeomorphism from $E(H_1)$ to $E(H_2)$.
\end{theorem}

An annulus is \textit{essential} if $A$ is incompressible and not
boundary-parallel.
To prove Theorem~\ref{thm:hdb-knots and their exteriors}, we give some
lemmas.

\begin{lemma}[{\cite[15.26 Lemma]{BurdeZieschang85}}]
\label{lem:annulus in E(K)}
Let $K$ be a knot in $S^3$.
If $E(K)$ contains a essential annulus $A$, then either
\begin{enumerate}
\item $K$ is a composite knot and $A$ can be extended to a decomposing sphere for $K$,
\item $K$ is a torus knot and $A$ can be extended to an unknotted torus or
\item $K$ is a cable knot and $A$ is the cabling annulus.
\end{enumerate}
\end{lemma}


\begin{lemma}[{\cite[Lemma~3.2]{Kobayashi84}}]
\label{lem:annulus in V_2}
If $A$ is a essential annulus in a genus two handlebody $W$, then either
\begin{enumerate}
\item
$A$ cuts $W$ into a solid torus $W_1$ and a genus two handlebody $W_2$
and there is a complete system of meridian disks $\{D_1,D_2\}$ of $W_2$
such that $D_1\cap A=\emptyset$ and $D_2\cap A$ is an essential arc in
$A$, or
\item
$A$ cuts $W$ into a genus two handlebody $W'$ and there is a complete
system of meridian disks $\{D_1,D_2\}$ of $W'$ such that $D_1\cap A$ is
an essential arc in $A$.
\end{enumerate}
\end{lemma}

We say that an annulus $A$ is obtained from a knotted handle decomposing
sphere $S$ for a handlebody-knot $H$ when $A=S\cap E(H)$.

\begin{lemma} \label{lem:cabling or khds}
Let $H$ be an irreducible genus two handlebody-knot with a knotted
handle decomposing sphere $S$ bounding $(B,K;H)$ such that $B$ contains
all spheres in a maximal unnested set of knotted handle decomposing
spheres for $H$.
Suppose that $E(H\cup B)$ is a handlebody and that
$H-\operatorname{int}B$ is a nontrivial handlebody-knot.
Then any essential separating annulus in $E(H)$ is isotopic to either a
cabling annulus for $H-\operatorname{int}B$ or an annulus obtained from
a knotted handle decomposing sphere for $H$.
\end{lemma}

\begin{proof}
Let $A'$ be a essential separating annulus in $E(H)$.
Assuming that $A'$ cannot be obtained from a knotted handle decomposing
sphere for $H$, we show that $A'$ is a cabling annulus for
$H-\operatorname{int}B$.
Put $A=S\cap E(H)$ and $W=E(H\cup B)$.
We may assume that $A\cap A'$ consists of essential arcs or loops in
both $A$ and $A'$, and that $|A\cap A'|$ is minimal up to isotopy.
As the proof of Lemma~\ref{lem:disjoint khds}, we may assume that
$A\cap A'$ consists of essential loops.

If $\partial A'$ is contained in $B$, then $A'$ is an annulus obtained
from a knotted handle decomposing sphere for $H$, since each loop of
$\partial A'$ is parallel to $\partial(S\cap H)$.
Hence there is a loop $C$ of $\partial A'$ contained in $W$.

Suppose $A\cap A'\neq\emptyset$.
Let $F$ be the outermost subannulus on $A'$ containing $C$, which is an
annulus properly embedded in $W$.
Since $A'$ is incompressible in $E(H)$, $F$ is incompressible in $W$.
By the minimality of $|A\cap A'|$, $F$ is not boundary-parallel in $W$.
Let $D$ be a disk in $E(H-\operatorname{int}B)$ such that $D\cap W=F$
and $D\cap B$ is a disk $D_0$ in $B$.
If $C$ is essential in $\partial(H-\operatorname{int}B)$, then
$E(H-\operatorname{int}B)$ is boundary-reducible, which implies that
$H-\operatorname{int}B$ is trivial, a contradiction.
Hence $C$ is inessential in $\partial(H-\operatorname{int}B)$.
Let $D'$ be the disk in $\partial(H-\operatorname{int}B)$ such that
$\partial D'=C$.
Let $D_1,D_2$ be the disks such that $S\cap H=D_1\cup D_2$.
If $C$ is parallel to $\partial D_0$ on $\partial(H\cup B)$, then $F$ is
an annulus obtained from a knotted handle decomposing sphere for the
trivial genus two handlebody-knot $H\cup B$, a contradiction.
Thus $D_1,D_2\subset D'$ or $(D_1\cup D_2)\cap D'=\emptyset$, which
contradicts that the $2$-sphere $S'=D'\cup D$ separates $D_1$ and $D_2$,
where $S'$ is slightly isotoped so that $D'$ is properly embedded in
$H-\operatorname{int}B$.
Hence $A\cap A'=\emptyset$, which implies that $A'\subset W$.

The annulus $A'$ is incompressible in $W$, since it is incompressible in
$E(H)$.
If $A'$ is boundary-parallel in $W$, then $A'$ is parallel to $A$ and is
obtained from a knotted handle decomposing sphere for $H$, since $A'$ is
not boundary-parallel in $E(H)$.
Hence $A'$ is essential in the genus two handlebody $W$.

By Lemma~\ref{lem:annulus in V_2}, the separating annulus $A'$ cuts $W$
into a solid torus $W_1$ and a genus two handlebody $W_2$ so that $A'$
winds around $W_1$ at least twice.
If $A$ is contained in $\partial W\cap W_1$, then by attaching a
2-handle $N(D)$ to the solid torus $W_1$, we have a once punctured lens
space $L(p,q)$ ($p\ge 2$), where $D$ is a component of $S\cap H$.
This contradicts Alexander's theorem~\cite{Alexander24}.
Thus $A$ is contained in $\partial W\cap W_2$ and $A'$ cuts $W\cup B$
into $W_1$ and $W_2\cup B$.

Suppose that $A'$ is compressible in $W\cup B$.
Let $D$ be a compressing disk for $A'$ in $W\cup B$.
Then $D$ is contained in $W_2\cup B$, since $A'$ is incompressible in
$W$.
By attaching a 2-handle $N(D)$ to the solid torus $W_1$, we have a once
punctured lens space $L(p,q)$ ($p\ge 2$), a contradiction.
Thus $A'$ is incompressible in $W\cup B$.
Suppose that $A'$ is boundary-parallel in $W\cup B$.
Since $A'$ is not boundary-parallel in $W$, $W_2\cup B$ is a solid torus
$A'\times I$.
Then the solid torus $W_1$ is isotopic to
$W\cup B=E(H-\operatorname{int}B)$, which implies that
$H-\operatorname{int}B$ is trivial, a contradiction.
Thus $A'$ is not boundary-parallel in $W\cup B$.
Therefore $A'$ is essential in $W\cup B=E(H-\operatorname{int}B)$, which
is the exterior of the tunnel number one knot represented by the core
curve of $H-\operatorname{int}B$.
By Lemma~\ref{lem:annulus in E(K)}, $A'$ is a cabling annulus for
$H-\operatorname{int}B$, where we note that a tunnel number one knot is
prime.
\end{proof}

\begin{proof}[Proof of Theorem~\ref{thm:hdb-knots and their exteriors}]
If $H_1$ and $H_2$ are equivalent, then there is an orientation
preserving self-homeomorphism of $S^3$ which sends $H_1$ to $H_2$, which
gives an orientation preserving homeomorphism from $E(H_1)$ to
$E(H_2)$.

Suppose that there is an orientation preserving homeomorphism $f$ from
$E(H_1)$ to $E(H_2)$.
Since any cabling annulus cuts off a solid torus from $E(H_2)$, it
follows from Lemma~\ref{lem:cabling or khds} that
$f(S_1\cap E(H_1))=S_2\cap E(H_2)$.
Since $E(H_i-\operatorname{int}B_i)$ and $B_i-\operatorname{int}H_i$ are
exteriors of knots, by the Gordon-Luecke theorem~\cite{GordonLuecke89},
both of the restrictions of $f$ to $E(H_1-\operatorname{int}B_1)$ and
$B_1-\operatorname{int}H_1$ are extended to homeomorphisms of $S^3$.
Hence $f$ can be extended to a homeomorphism $\hat{f}$ of $S^3$ such
that $\hat{f}(S_1)=S_2$ and $\hat{f}(H_1)=H_2$.
\end{proof}

\begin{example}
By Example~\ref{ex:5_4,6_14,6_15}, neither $6_{15}$ nor $6_{15}^*$ is
equivalent to $6_{14}$.
We recall that each of them has a unique knotted handle decomposing
sphere.
By Theorem~\ref{thm:hdb-knots and their exteriors}, there is no
orientation preserving/reversing homeomorphism from $E(6_{14})$ to
$E(6_{15})$.
Hence $E(6_{14})$ and $E(6_{15})$ are not homeomorphic.
\end{example}

\bibliographystyle{amsplain}

\end{document}